\documentclass[10pt]{amsart}
\usepackage{amssymb,amstext,amsmath,amscd,amsthm,amsfonts,enumerate,graphicx,latexsym,stmaryrd,multicol,geometry}
\usepackage[usenames]{color}
\usepackage[all]{xy}
\newtheorem{thm}{Theorem}[section]
\newtheorem{lem}[thm]{Lemma}
\newtheorem{prop}[thm]{Proposition}

\theoremstyle{definition}
\newtheorem{dfn}[thm]{Definition}

\newtheorem{rem}[thm]{Remark}
\newtheorem{conv}[thm]{Convention}

\newtheorem{ex}[thm]{Example}

\theoremstyle{remark}
\newtheorem{claim}{Claim}
\newtheorem*{claim*}{Claim}
\newtheorem*{ac}{Acknowlegments}

\numberwithin{equation}{thm}
\def\A{\mathrm{A}}
\def\a{\mathfrak{a}}
\def\ann{\operatorname{Ann}}
\def\ass{\operatorname{Ass}}
\def\b{\mathfrak{b}}
\def\c{\mathfrak{c}}
\def\cm{\operatorname{CM}}
\def\D{\operatorname{D}}
\def\d{\operatorname{\mathsf{D}}}
\def\db{\operatorname{\mathsf{D^b}}}
\def\depth{\operatorname{depth}}
\def\F{\mathrm{F}}
\def\ge{\geqslant}
\def\gor{\operatorname{Gor}}
\def\H{\operatorname{H}}
\def\height{\operatorname{ht}}
\def\Hom{\operatorname{Hom}}
\def\id{\mathrm{id}}
\def\K{\operatorname{K}}
\def\le{\leqslant}
\def\ltensor{\otimes^\mathbf{L}}
\def\m{\mathfrak{m}}
\def\mcm{\operatorname{MCM}}
\def\Min{\operatorname{Min}}
\def\Mod{\operatorname{Mod}}
\def\mod{\operatorname{mod}}
\def\N{\mathbb{N}}
\def\p{\mathfrak{p}}
\def\q{\mathfrak{q}}

\def\rhom{\operatorname{\mathbf{R}Hom}}
\def\rg{\mathbf{R}\Gamma}
\def\spec{\operatorname{Spec}}
\def\supp{\operatorname{Supp}}
\def\T{\mathsf{T}}
\def\V{\operatorname{V}}
\def\X{\mathsf{X}}
\def\Y{\mathsf{Y}}

\def\ZZ{\mathbb{Z}}
\begin{document}
\allowdisplaybreaks
\title[Faltings' annihilator theorem and $t$-structures of derived categories]{Faltings' annihilator theorem\\
and $t$-structures of derived categories}
\author{Ryo Takahashi}
\address{Graduate School of Mathematics, Nagoya University, Furocho, Chikusaku, Nagoya 464-8602, Japan}
\email{takahashi@math.nagoya-u.ac.jp}
\urladdr{https://www.math.nagoya-u.ac.jp/~takahashi/}
\thanks{The author was partly supported by JSPS Grant-in-Aid for Scientific Research 19K03443}
\begin{abstract}
In this paper, we prove Faltings' annihilator theorem for complexes over a CM-excellent ring.
As an application, we give a complete classification of the $t$-structures of the bounded derived category of finitely generated modules over a CM-excellent ring of finite Krull dimension.
\end{abstract}
\maketitle
\section{Introduction}

Let $R$ be a commutative noetherian ring.
Following \v{C}esnavi\v{c}ius \cite{C}, we say that $R$ is {\em CM-excellent} if it satisfies the three conditions below, which have been studied deeply by Kawasaki \cite{K0,K,K2}.
\begin{itemize}
\item
The ring $R$ is universally catenary.
\item
The formal fibers of the localization of $R$ at each prime ideal are Cohen--Macaulay.
\item
The Cohen--Macaulay locus of each finitely generated $R$-algebra is Zariski-open.
\end{itemize}
Typical examples of a CM-excellent ring include an excellent ring, more generally an acceptable ring in the sense of Sharp \cite{S}, and a homomorphic image of a Cohen--Macaulay ring \cite{K2}.
In particular, the ring $R$ is CM-excellent if it possesses a dualizing complex, since the existence of a dualizing complex is equivalent to the condition that the ring is a homomorphic image of a Gorenstein ring of finite Krull dimension \cite{K0}.

Let $\db(R)$ stand for the bounded derived category of finitely generated $R$-modules.
The first main result of this paper is the following theorem, which is {\em Faltings' annihilator theorem} for complexes.

\begin{thm}[Theorem \ref{11}]\label{20}
Let $R$ be a CM-excellent ring.
Let $Y$ and $Z$ be specialization-closed subsets of $\spec R$, and let $n$ be an integer.
Then the following two conditions are equivalent for each $X\in\db(R)$.
\begin{enumerate}[\rm(1)]
\item
For all prime ideals $\p$ and $\q$ of $R$ with $Z\ni\p\supseteq\q\notin Y$, one has the inequality $\height\p/\q+\depth X_\q\ge n$.
\item
There exists an ideal $\b$ of $R$ such that $\V(\b)\subseteq Y$ and $\b\H_Z^{<n}(X)=0$.
\end{enumerate}
\end{thm}

\noindent
If we restrict Theorem \ref{20} to the case where the complex $X$ is a module, then it is the same as the main result of \cite{K}, which extends a lot of previous results with additional assumptions, including Faltings' original one \cite{F}; see \cite{K} for more details.
The main result of \cite{DZ} shows the assertion of Theorem \ref{20} under the stronger assumptions that $Y$ contains $Z$ and that $R$ possesses a dualizing complex.
The latter assumption is to use the local duality theorem; it does play an essential role in the proof of the result of \cite{DZ}.

As an application of Theorem \ref{20}, we obtain the second main result of this paper: the following theorem provides a complete classification of the {\em $t$-structures} (in the sense of Be\u{\i}linson, Bernstein and Deligne \cite{BBD}) of the triangulated category $\db(R)$ in terms of certain filtrations by specialization-closed subsets of $\spec R$.

\begin{thm}[Theorem \ref{12}]\label{21}
Let $R$ be a CM-excellent ring with finite Krull dimension.
Then the aisles in $\db(R)$ bijectively correspond to the sp-filtrations of $\spec R$ satisfying the weak Cousin condition.
\end{thm}

\noindent
The mutually inverse bijections giving the one-to-one correspondence in this theorem can be described explicitly; see Theorem \ref{12}.
The main result of \cite{AJS} shows the assertion of Theorem \ref{21} under the stronger assumption that the ring $R$ admits a dualizing complex.
This assumption is, again, to apply the local duality theorem, and in fact, local duality plays a key role in the proof of the result of \cite{AJS}.

The organization of the present paper is as follows.
Section 2 is devoted to stating several preliminary definitions and results, including some observations on CM-excellence.
In Section 3, we shall prove Faltings' annihilator theorem for complexes, that is to say, we shall give a proof of Theorem \ref{20} stated above.
The proof is done by reducing to the case of modules and invoking Kawasaki's result \cite{K}, but the reduction step turns out to be quite complicated and subtle, requiring various techniques.
In Section 4, we recall the definition of an sp-filtration of $\spec R$ satisfying the weak Cousin condition, and interpret it in terms of a certain function on $\spec R$ which we call a $t$-function.
In the final Section 5, we shall apply Theorem \ref{20} to obtain Theorem \ref{21}, where a classification of $t$-structures by the $t$-functions on $\spec R$ is provided as well.

\section{Preliminaries}

This section consists of basic definitions, fundamental properties, simple observations that are used in later sections.
We begin with our convention adopted throughout the present paper.

\begin{conv}
We assume that all rings are commutative and noetherian, and all complexes are cochain ones.
Let $R$ be a (commutative noetherian) ring.
We denote by $\Mod R$ the category of $R$-modules, by $\d(R)$ the derived category of $\Mod R$, by $\mod R$ the category of finitely generated $R$-modules, and by $\db(R)$ the bounded derived category of $\mod R$.
The set of nonnegative integers is denoted by $\N$.
\end{conv}

We need the following standard numerical invariants for objects of $\d(R)$.

\begin{dfn}
The {\em supremum} and the {\em infimum} of $X\in\d(R)$ are defined respectively by $\sup X=\sup\{i\in\ZZ\mid\H^i(X)\ne0\}$ and $\inf X=\inf\{i\in\ZZ\mid\H^i(X)\ne0\}$.
Note that $\sup X$ and $\inf X$ are elements of $\ZZ\cup\{\pm\infty\}$.
\end{dfn}

Taking (soft) truncations gives rise to a certain series of exact triangles in $\db(R)$.

\begin{rem}\label{4}
Let $X$ be a nonzero object of $\db(R)$.
Then there exists a series
$$
\{\,X_{i+1}\to X_i\to\H^{s_i}(X_i)[-s_i]\rightsquigarrow\,\}_{i=0}^n
$$
of exact triangles in $\db(R)$ such that $n\in\N$, $X_0=X$, $X_{n+1}=0$, $s_i=\sup X_i\in\ZZ$, $\H^j(X_i)=\H^j(X)$ for all $j\le s_i$, and $\sup X=s_0>\cdots>s_n=\inf X$.
\end{rem}

Next we recall the definitions of a specialization-closed subset and a local cohomology functor.

\begin{dfn}
\begin{enumerate}[(1)]
\item
A subset $W$ of $\spec R$ is called {\em specialization-closed} if $\V(\p)\subseteq W$ for all $\p\in W$.
This is equivalent to saying that $W$ is a union of closed subsets of $\spec R$ (in the Zariski topology).
\item
Let $W$ be a specialization-closed subset of $\spec R$.
For an $R$-module $M$ we denote by $\Gamma_W(M)$ the set of elements $x\in M$ such that $\supp(Rx)\subseteq W$.
Then $\Gamma_W(M)$ is a submodule of $M$, and we get a left-exact additive covariant functor $\Gamma_W:\Mod R\to\Mod R$, which is called the {\em $W$-torsion functor}.
Using $K$-injective resolutions, one gets the right derived functor $\rg_W:\d(R)\to\d(R)$ of $\Gamma_W$.
For each $i$ set $\H_W^i=\H^i\rg_W:\d(R)\to\Mod R$ and call it the {\em $i$th local cohomology} functor with respect to the specialization-closed subset $W$.
When $I$ is an ideal of $R$ and $W=\V(I)$, the functor $\H_W^i$ coincides with the usual $i$th local cohomology functor $\H_I^i$ with respect to the ideal $I$.
\end{enumerate}
\end{dfn}

We recall the definition of the depth of an object of $\db(R)$ for a local ring $R$.

\begin{dfn}
Let $(R,\m,k)$ be a local ring.
Let $X\in\db(R)$.
Then one has the equality $\inf\rhom_R(k,X)=\inf\rg_\m(X)$; see \cite[Theorem I]{FI}.
This value is called the {\em depth} of $X$ and denoted by $\depth X$.
\end{dfn}

The following statements are straightforward from the definition of a depth.

\begin{rem}\label{1}
Let $R$ be a local ring.
Let $W$ be a specialization-closed subset of $\spec R$.
\begin{enumerate}[(1)]
\item
For each $X\in\db(R)$ and $n\in\ZZ$ one has $\depth X[n]=\depth X-n$.
\item
Let $X\to Y\to Z\to X[1]$ be an exact triangle in $\db(R)$.
Then $\depth X\ge\inf\{\depth Y,\,\depth Z+1\}$, $\depth Y\ge\inf\{\depth X,\,\depth Z\}$ and $\depth Z\ge\inf\{\depth Y,\,\depth X-1\}$.
\end{enumerate}
\end{rem}

Now we recall the definition of a Cohen--Macaulay locus, and introduce and explain CM-excellence.

\begin{dfn}
\begin{enumerate}[(1)]
\item
We denote by $\cm(R)$ the {\em Cohen--Macaulay locus} of $R$, that is, the set of prime ideals $\p$ of $R$ such that the local ring $R_\p$ is Cohen--Macaulay.
\item
Following \v{C}esnavi\v{c}ius \cite{C}, we say that $R$ is {\em CM-excellent} if $R$ satisfies {\em Kawasaki's three conditions}:
\begin{enumerate}[\quad(C1)]
\item
The ring $R$ is universally catenary.
\item
The formal fibers of $R_\p$ are Cohen--Macaulay for every $\p\in\spec R$.
\item
For each finitely generated $R$-algebra $S$ the subset $\cm(S)$ of $\spec S$ is open.
\end{enumerate}
\end{enumerate}
\end{dfn}

\begin{rem}\label{22}
\begin{enumerate}[(1)]
\item
The ring $R$ is CM-excellent if and only if the scheme $\spec R$ is CM-excellent in the sense of \cite[Definition 1.2]{C}.
In fact, the ``only if'' part holds since the condition \cite[Definition 1.2(2)]{C} for $\spec R$ means that $\cm(R/\p)$ is an open subset of $\spec R/\p$ for all $\p\in\spec R$.
To show the ``if'' part, let $S$ be a finitely generated $R$-algebra.
Suppose that $\spec R$ is CM-excellent.
Then $\spec S$ is CM-excellent as well, by \cite[Remark 1.5]{C}.
Hence for any $\q\in\spec S$ the subset $\cm(S/\q)$ of $\spec S/\q$ is open, and nonempty since it contains the zero ideal.
It follows by \cite[Theorem 24.5]{M} that $\cm(S)$ is an open subset of $\spec S$.
\item
If $R$ is acceptable in the sense of Sharp \cite{S}, then it is CM-excellent.
Indeed, let $S$ be a finitely generated $R$-algebra.
Then for each $\q\in\spec S$ the Gorenstein locus $\gor(S/\q)$ is open and nonempty (as it contains the zero ideal).
Since $\cm(S/\q)$ contains $\gor(S/\q)$, it follows by \cite[Theorem 24.5]{M} that $\cm(S)$ is open.
\item
If $R$ is excellent, then it is acceptable by \cite[Corollary 1.5]{GM}, and hence it is CM-excellent by (2).
\item
If $R$ is a homomorphic image of a Cohen--Macaulay ring, it is CM-excellent by \cite[Theorem 1.3]{K2}.
Hence, $R$ is a CM-excellent ring of finite Krull dimension if it has a dualizing complex by \cite[Corollary 1.4]{K0}.
\item
When $R$ has a Cohen--Macaulay module with full support, it is CM-excellent by \cite[Remark 1.4]{C} and (1).
\end{enumerate}
\end{rem}

\section{Faltings' annihilator theorem for complexes}

The purpose of this section is to prove Faltings' annihilator theorem for complexes over a CM-excellent ring, which is Theorem \ref{11}.
All the other things (except Remark \ref{13}) stated in the section are to achieve this purpose.
As is seen below, to show the theorem we use a reduction to the case of (shifts of) modules, which is rather complicated and subtle, and requires several techniques of not only local cohomology and depths of complexes, but also truncation and annihilation of complexes, Koszul complexes, torsion submodules with respect to specialization-closed subsets, Zariski-openness of loci, associated prime ideals and prime avoidance.

We begin with introducing some notation about specialization-closed subsets.

\begin{dfn}
Let $W$ be a specialization-closed subset of $\spec R$.
\begin{enumerate}[(1)]
\item
Let $P$ be a prime ideal of $R$.
We denote by $W_P$ the set of prime ideals of $R_P$ having the form $\p R_P$ with $P\supseteq\p\in W$.
This is a specialization-closed subset of $\spec R_P$.
Note that $W_P\ne\emptyset$ if and only if $P\in W$.
\item
Let $I$ be a proper ideal of $R$.
We denote by $W/I$ the set of prime ideals of $R/I$ having the form $\p/I$ with $I\subseteq\p\in W$.
This is a specialization-closed subset of $\spec R/I$. 
\end{enumerate}
\end{dfn}

The lemma below is used to show one of the implications of the equivalence given in our theorem.

\begin{lem}\label{15}
Let $W$ be a specialization-closed subset of $\spec R$.
Let $n\in\ZZ$ and $X\in\db(R)$.
\begin{enumerate}[\rm(1)]
\item
Suppose that $(R,\m)$ is local and $W$ is nonempty.
If $\H_W^i(X)=0$ for all integers $i<n$, then $\depth X\ge n$.
\item
Let $\p$ be a prime ideal of $R$.
If $\p$ belongs to $W$ and $\H_W^i(X)_\p=0$ for all integers $i<n$, then $\depth X_\p\ge n$.
\end{enumerate}
\end{lem}

\begin{proof}
(1) As $R$ is local and $W$ is nonempty, $\m$ is in $W$.
We may assume $X\ncong0$ in $\db(R)$.
Put $t=\inf X\in\ZZ$.
Take an injective resolution $E=(0\to E^t\to E^{t+1}\to\cdots)$ of $X$.
As $0=\H_W^i(X)=\H^i(\Gamma_W(E))$ for all $i<n$, the induced sequence $0\to\Gamma_W(E^t)\to\cdots\to\Gamma_W(E^n)$ is exact.
All the terms are injective modules, so the sequence is split exact.
Hence the induced sequence $0\to\Gamma_\m(\Gamma_W(E^t))\to\cdots\to\Gamma_\m(\Gamma_W(E^n))$ is (split) exact as well, which is isomorphic to the sequence $0\to\Gamma_\m(E^t)\to\cdots\to\Gamma_\m(E^n)$ as $\m\in W$.
Thus $\H_\m^{<n}(X)=0$.

(2) We have $0=\H_W^i(X)_\p=\H_{W_\p}^i(X_\p)$ for all $i<n$; the last equality follows by \cite[Corollary 3.5(1)]{DZ}. 
Since $\p$ belongs to $W$, the set $W_\p$ is nonempty.
Using (1), we obtain the desired inequality $\depth X_\p\ge n$.
\end{proof}

The next lemma plays an important role in the proof of the other implication of the equivalence in the theorem.
Thanks to this result, given a specialization-closed subset $Y$ and a complex $X$, we can replace $X$ with another complex $X'$ whose support is deeply involved with $Y$.
We denote by $\K(-)$ the Koszul complex.

\begin{lem}\label{3}
Let $X$ be an object of $\db(R)$.
The following statements hold.
\begin{enumerate}[\rm(1)]
\item
There is an ideal $I=(x_1,\dots,x_r)$ of $R$ such that $\supp X=\V(I)$ and $(X\xrightarrow{x_i}X)=0$ in $\db(R)$ for all $i$.
\item
Let $y_1,\dots,y_s$ be a sequence of elements of $R$, and put $X'=X\ltensor_R\K(y_1,\dots,y_s)$.
\begin{enumerate}[\rm(a)]
\item
Let $\p\supseteq\q$ be prime ideals of $R$.
Let $n\in\ZZ$.
If $\height\p/\q+\depth X_\q\ge n$, then $\height\p/\q+\depth X'_\q\ge n-s$.
\item
Let $Y,Z$ be a specialization-closed subset of $\spec R$.
Let $I$ be an ideal as in (1).
Suppose that the ideal $J=(y_1,\dots,y_s)$ contains $I$ and satisfies $\Gamma_{Y/I}(R/I)=J/I$.
If there is an ideal $\c$ of $R$ with $\V(\c)\subseteq Y$ and $\c\H_Z^{<(n-s)}(X')=0$, then there is an ideal $\b$ of $R$ with $\V(\b)\subseteq Y$ and $\b\H_Z^{<n}(X)=0$.
\end{enumerate}
\end{enumerate}
\end{lem}

\begin{proof}
(1) The assertion is immediately follows from \cite[Proposition 2.3(2)]{dm}.

(2) Put $X_i=X\ltensor_R\K(y_1,\dots,y_i)$ for each $0\le i\le s$.
Then $X_0=X$ and $X_s=X'$.
For $1\le i\le s$, applying $X_{i-1}\ltensor_R-$ to the exact triangle $R\xrightarrow{y_i}R\to\K(y_i)\rightsquigarrow$, we get an exact triangle $X_{i-1}\xrightarrow{y_i}X_{i-1}\to X_i\rightsquigarrow$.

(a) For $1\le i\le s$ there is an exact triangle $(X_{i-1})_\q\xrightarrow{y_i}(X_{i-1})_\q\to(X_i)_\q\rightsquigarrow$, which induces $\depth(X_i)_\q\ge\inf\{\depth(X_{i-1})_\q,\depth(X_{i-1})_\q-1\}=\depth(X_{i-1})_\q-1$ by Remark \ref{1}(2).
Hence $\depth X'_\q\ge\depth X_\q-s$.

(b) Set $\b_i=\c(I:y_s)\cdots(I:y_{i+1})$ for every integer $0\le i\le s$.
Then, as $\b_s=\c$, we have $\b_s\H_Z^{<(n-s)}(X_s)=0$.
Fix an integer $1\le i\le s$ and assume that $\b_i\H_Z^{<(n-i)}(X_i)=0$.
For each $j\in\ZZ$ there exists an exact sequence
\begin{equation}\label{2}
\H_Z^{j-i}(X_i)\to\H_Z^{j-i+1}(X_{i-1})\xrightarrow{y_i}\H_Z^{j-i+1}(X_{i-1}).
\end{equation}
Note that $(I:y_i)\cdot y_i\H_Z^{j-i+1}(X_{i-1})$ is contained in $I\H_Z^{j-i+1}(X_{i-1})$.
By (1), we have $(X\xrightarrow{x_p}X)=0$ in $\db(R)$ for every $1\le p\le r$.
Applying the functor $\H_Z^{j-i+1}(-\ltensor_R\K(y_1,\dots,y_{i-1}))$, we observe that $(\H_Z^{j-i+1}(X_{i-1})\xrightarrow{x_p}\H_Z^{j-i+1}(X_{i-1}))=0$ in $\Mod R$ for every $1\le p\le r$, which implies $I\H_Z^{j-i+1}(X_{i-1})=0$ for all $j\in\ZZ$.
By assumption, we have $\b_i\H_Z^{j-i}(X_i)=0$ for all $j<n$.
It follows from \eqref{2} that $\b_{i-1}=\b_i(I:y_i)$ kills $\H_Z^{j-i+1}(X_{i-1})$ for any $j<n$.
By induction on $s-i$, we get $\b_0\H_Z^j(X_0)=0$ for any $j<n$, i.e., $\b_0\H_Z^{<n}(X)=0$.

It remains to show that $\V(\b_0)$ is contained in $Y$, and for this it suffices to check that $\V(I:y_i)$ is contained in $Y$ for all $1\le i\le s$.
Take any $\p\in\V(I:y_i)$.
Then there are inclusions $I\subseteq I:_Ry_i\subseteq\p$, which induce $0:_{R/I}\overline{y_i}\subseteq\p/I$.
Hence $\p/I\in\V(0:_{R/I}\overline{y_i})=\supp_{R/I}((R/I)\overline{y_i})\subseteq Y/I$, where the last inclusion follows from the fact that $\overline{y_i}\in J/I=\Gamma_{Y/I}(R/I)$.
Therefore the prime ideal $\p$ belongs to $Y$, and we are done.
\end{proof}

Here we introduce a certain locus in $\spec R$ of a finitely generated $R$-module.
This is necessary in the proof of the theorem to make the depth of a localized module high enough.

\begin{dfn}
Let $M$ be a finitely generated $R$-module.
The {\em maximal Cohen--Macaulay locus} $\mcm_R(M)$ of $M$ is defined as the set of prime ideals $\p$ of $R$ such that the $R_\p$-module $M_\p$ is maximal Cohen--Macaulay, i.e., $\depth M_\p\ge\dim R_\p$.
This inequality is equivalent to saying that $M_\p=0$ or $M_\p\ne0$ and $\depth M_\p=\dim R_\p$.
\end{dfn}

Now we can prove the main result of this section, which asserts that Faltings' annihilator theorem holds true for bounded complexes of finitely generated modules over a CM-excellent ring.

\begin{thm}\label{11}
Let $R$ be a CM-excellent ring.
Let $Y,Z$ be specialization-closed subsets of $\spec R$, and let $n$ be an integer.
Then the following two conditions are equivalent for each $X\in\db(R)$.
\begin{enumerate}[\rm(1)]
\item
For all prime ideals $\p,\q$ of $R$ with $Z\ni\p\supseteq\q\notin Y$ there is an inequality $\height\p/\q+\depth X_\q\ge n$.
\item
There exists an ideal $\b$ of $R$ such that $\V(\b)\subseteq Y$ and $\b\H_Z^{<n}(X)=0$.
\end{enumerate}
\end{thm}

\begin{proof}
(2) $\Rightarrow$ (1):
Put $Y'=Y\cup Z$.
This is a specialization-closed subset of $\spec R$.
As $\V(\b)\subseteq Y\subseteq Y'$ and $Z\subseteq Y'$, we can apply \cite[Theorem 4.5]{DZ} to see that $\height\p/\q+\depth X_\q\ge n$ for all $Z\ni\p\supseteq\q\notin Y'$.

Now, fix prime ideals $\p,\q$ such that $Z\ni\p\supseteq\q\notin Y$.
If $\q$ is not in $Z$, then we have $Z\ni\p\supseteq\q\notin Y'$ and get $\height\p/\q+\depth X_\q\ge n$.
Suppose $\q\in Z$.
Since $\q$ is not in $Y$, it does not contain $\b$.
Hence $0=(\b\H_Z^i(X))_\q=\H_Z^i(X)_\q$ for all $i<n$.
Lemma \ref{15}(2) implies $\depth X_\q\ge n$, and therefore $\height\p/\q+\depth X_\q\ge n$.

(1) $\Rightarrow$ (2):
Assume that (1) holds.
We shall deduce (2) by noetherian induction on $\supp X$.
If $\supp X=\emptyset$, then $X\cong0$ in $\db(R)$, and (2) holds by letting $\b=R$.
Let $\supp X\ne\emptyset$.
Then $X\ncong0$ in $\db(R)$.
Use Lemma \ref{3}(1) to find an ideal $I=(x_1,\dots,x_r)$ of $R$ such that $\supp X=\V(I)$ and $(X\xrightarrow{x_i}X)=0$ in $\db(R)$ for all $i$.
Choose an ideal $J=(y_1,\dots,y_s)$ of $R$ such that $I\subseteq J$ and $\Gamma_{Y/I}(R/I)=J/I$.
Set $X'=X\ltensor_R\K(y_1,\dots,y_s)\in\db(R)$.
We have $\supp X'=\supp X\cap\V(J)=\V(I)\cap\V(J)=\V(J)$, where the first equality follows from \cite[Lemma 1.9(4) and Proposition 2.3(3)]{dm}.
According to Lemma \ref{3}(2a), we have
\begin{equation}\label{5}
\height\p/\q+\depth X'_\q\ge n-s\ \ \text{for all}\ \ \p,\q\in\spec R\ \ \text{such that}\ \ Z\ni\p\supseteq\q\notin Y.
\end{equation}
In view of Lemma \ref{3}(2b), we will be done if we find an ideal $\c$ of $R$ such that $\V(\c)\subseteq Y$ and $\c\H_Z^{<(n-s)}(X')=0$.
To show this, we may assume $X'\ncong0$ in $\db(R)$, and then $J$ is a proper ideal of $R$.
It follows from \cite[Remark 2.7]{dec} that $(X'\xrightarrow{y_j}X')=0$ in $\db(R)$ for all $j$.
Hence $(\H(X')\xrightarrow{y_j}\H(X'))=0$ in $\mod R$ for all $j$, which means that $J$ annihilates $\H(X')$, or in other words, $\H(X')\in\mod R/J$.
As $R$ is CM-excellent, any finitely generated $R$-algebra has open Cohen--Macaulay locus.
In particular, for every prime ideal $P/J$ of $\spec R/J$ the locus $\cm((R/J)/(P/J))=\cm(R/P)$ is open and nonempty as it contains the zero ideal.
The ring $R/J$ satisfies \cite[(5.0.1)]{Ki}, and $\mcm_{R/J}(\H(X'))$ is an open subset of $\spec R/J$ by \cite[Corollary 5.5(3)]{Ki}.
There is an ideal $\a$ of $R$ containing $J$ such that $\mcm_{R/J}(\H(X'))=\D(\a/J)$.
Note that $\mcm_{R/J}(\H(X'))$ contains $\Min R/J$.
Prime avoidance gives an element $a\in\a$ with $a\notin\bigcup_{P\in\Min_RR/J}P$.
Apply Remark \ref{4} to $X'$ to get a series
\begin{equation}\label{8}
\{\,X_{i+1}\to X_i\to\H^{s_i}(X')[-s_i]\rightsquigarrow\,\}_{i=0}^e
\end{equation}
of exact triangles in $\db(R)$ such that $X_0=X'$, $X_{e+1}=0$, $s_i=\sup X_i\in\ZZ$, $\H^j(X_i)=\H^j(X')$ for all $j\le s_i$, and $\sup X'=s_0>\cdots>s_e=\inf X'$.
We establish a claim.

\begin{claim}\label{9}
For all integers $0\le i\le e$ and for all prime ideals $\p,\q$ of $R$ such that $Z\ni\p\supseteq\q\notin Y\cup\V(a)$, one has the inequality $\height\p/\q+\depth\H^{s_i}(X')_\q\ge n-s-s_i$.
\end{claim}

\begin{proof}[Proof of Claim \ref{9}]
Fix an integer $0\le i\le e$ and prime ideals $\p,\q$ with $Z\ni\p\supseteq\q\notin Y\cup\V(a)$.
The zero module has depth $\infty$, so that we may assume $\H^{s_i}(X')_\q\ne0$.
Since $\H^{s_i}(X')$ is a direct summand of $\H(X')$,  we have $\q\in\supp\H^{s_i}(X')\subseteq\supp\H(X')\subseteq\V(J)$.
As $\q$ is not in $\V(a)$, we see that $\q/J\in\D(\a/J)=\mcm_{R/J}(\H(X'))$.
This shows the second inequality below, while the first holds as $\H^{s_i}(X')_\q$ is a direct summand of $\H(X')_\q$.
$$
\depth\H^{s_i}(X')_\q\ge\depth\H(X')_\q=\depth\H(X')_{\q/J}\ge\dim(R/J)_{\q/J}=\height\q/J.
$$
Since $\H^{s_i}((X_i)_\q)=\H^{s_i}(X_i)_\q=\H^{s_i}(X')_\q\ne0$, it holds that $(X_i)_\q\ncong0$ in $\db(R_\q)$, and we get $\depth(X_i)_\q\le\inf(X_i)_\q+\dim\H^{\inf(X_i)_\q}((X_i)_\q)$ by \cite[(1.4.4)]{Ch}.
We have $\inf(X_i)_\q\le\sup X_i=s_i$, while
$$
\dim\H^{\inf(X_i)_\q}((X_i)_\q)\le\dim\H(X_i)_\q\le\dim\H(X')_\q=\height(\q/\ann\H(X'))\le\height\q/J\le\depth\H^{s_i}(X')_\q.
$$
Here, the first two inequalities come from the fact that $\H^{\inf(X_i)_\q}((X_i)_\q)$ is a direct summand of $\H(X_i)_\q$, which is a direct summand of $\H(X')_\q$.
The inclusion $J\subseteq\ann\H(X')$ shows the third inequality.
Therefore,
\begin{equation}\label{6}
\depth(X_i)_\q\le\inf(X_i)_\q+\dim\H^{\inf(X_i)_\q}((X_i)_\q)\le s_i+\depth\H^{s_i}(X')_\q=\depth(\H^{s_i}(X')[-s_i])_\q,
\end{equation}
where the equality follows from Remark \ref{1}(1).
Applying Remark \ref{1}(2) to the exact triangle $(X_{i+1})_\q\to(X_i)_\q\to(\H^{s_i}(X')[-s_i])_\q\rightsquigarrow$ in $\db(R_\q)$, we observe that
$$
\depth(X_{i+1})_\q\ge\inf\{\depth(X_i)_\q,\,\depth(\H^{s_i}(X')[-s_i])_\q+1\}=\depth(X_i)_\q.
$$

Consequently, for each integer $0\le i\le e$ it holds that
\begin{equation}\label{7}
\depth(X_i)_\q\ge\depth(X_{i-1})_\q\ge\cdots\ge\depth(X_0)_\q=\depth X'_\q.
\end{equation}
It follows from \eqref{5}, \eqref{6} and \eqref{7} that for all integers $0\le i\le e$ there are inequalities
$$
\height\p/\q+\depth\H^{s_i}(X')_\q
\ge\height\p/\q+(\depth(X_i)_\q-s_i)
\ge\height\p/\q+(\depth X'_\q-s_i)
\ge n-s-s_i.
$$
Now the proof of the claim is completed.
\renewcommand{\qed}{$\square$}
\end{proof}

By Claim \ref{9}, we can apply \cite[Theorem 1.1]{K} to the module $\H^{s_i}(X')$ for $0\le i\le e$ to find an ideal $\b_i$ of $R$ with $\V(\b_i)\subseteq Y\cup\V(a)$ and $\b_i\H_Z^{<(n-s-s_i)}(\H^{s_i}(X'))=0$.
Hence $\b_i\H_Z^{<(n-s)}(\H^{s_i}(X')[-s_i])=0$.
Applying $\H_Z$ to the exact triangles \eqref{8} shows $\b\H_Z^{<(n-s)}(X')=0$, where $\b:=\b_0\b_1\cdots\b_e$.
We establish another claim.

\begin{claim}\label{add}
There exists an element $b\in\b$ such that $b\notin\bigcup_{P\in\Min_RR/J}P$.
\end{claim}

\begin{proof}[Proof of Claim \ref{add}]
Suppose that the ideal $\b$ is contained in some $P\in\Min_RR/J=\min\V(J)$.
Then $I\subseteq J\subseteq P$ and $P\in\V(\b)\subseteq Y\cup\V(a)$.
The choice of the element $a$ shows $P\in Y$, and hence $P/I\in Y/I$.
It is seen that
$$
P/I\in\min\V(J/I)=\Min_{R/I}((R/I)/(J/I))=\Min_{R/I}R/J\subseteq\ass_{R/I}R/J.
$$
Hence, $P/I$ belongs to the set $(\ass_{R/I}R/J)\cap Y/I$.
However, since there are equalities
$$
0=\Gamma_{Y/I}((R/I)/\Gamma_{Y/I}(R/I))=\Gamma_{Y/I}((R/I)/(J/I))=\Gamma_{Y/I}(R/J),
$$
the set $(\ass_{R/I}R/J)\cap Y/I$ is empty; see \cite[Proposition 3.2(1b)(3)]{cd}.
This contradition shows that none of the prime ideals $P\in\Min_RR/J$ contains $\b$.
Prime avoidance completes the proof of the claim.
\renewcommand{\qed}{$\square$}
\end{proof}

Since $b$ is in $\b$, we have $b\H_Z^{<(n-s)}(X')=0$.
Set $X''=X'\ltensor_R\K(b)\in\db(R)$.
Applying the functor $X'\ltensor_R-$ to the exact triangle $R\xrightarrow{b}R\to\K(b)\rightsquigarrow$ gives an exact triangle $X'\xrightarrow{b}X'\to X''\rightsquigarrow$.
There are equalities
$$
\supp X''=\supp X'\cap\V(b)=\V(J)\cap\V(b)=\V(J+(b))\subsetneq\V(J)\subseteq\V(I)=\supp X.
$$
Here, the first equality follows from \cite[Lemma 1.9(4) and Proposition 2.3(3)]{dm} again.
The strict inclusion there follows since we can take a minimal prime ideal $P$ of $J$, and then the choice of $b$ shows that $P$ does not contain $b$.
Let $\p,\q$ be prime ideals of $R$ such that $Z\ni\p\supseteq\q\notin Y$.
Then $\depth X''_\q\ge\depth X'_\q-1$ by Remark \ref{1}(2), and hence $\height\p/\q+\depth X''_\q\ge n-s-1$ by \eqref{5}.
Apply the induction hypothesis to $X''$ to find an ideal $\c$ of $R$ such that $\V(\c)\subseteq Y$ and $\c\H_Z^{<(n-s-1)}(X'')=0$.
The exact sequence $\H_Z^{i-1}(X'')\to\H_Z^i(X')\xrightarrow{b}\H_Z^i(X')$ is induced for each integer $i$, and the vanishings $b\H_Z^i(X')=\c\H_Z^{i-1}(X'')=0$ for all $i<n-s$ show that $\c\H_Z^i(X')=0$ for all $i<n-s$.
The proof of the theorem is now completed.
\end{proof}

We close the section by mentioning the relationships of the above theorem with results in the literature.

\begin{rem}\label{13}
\begin{enumerate}[(1)]
\item
If we restrict Theorem \ref{11} to the case where the complex $X\in\db(R)$ is a module, then the theorem asserts the same as \cite[Theorem 1.1]{K}.
\item
Taking Remark \ref{22}(4) into account, one sees that Theorem \ref{11} holds for any homomorphic image of a Cohen--Macaulay ring, and hence it holds for any ring admitting a dualizing complex.
\item
Theorem \ref{11} removes the assumption in \cite[Theorem 4.5]{DZ} that $Y$ contains $Z$ and weakens the assumtion there that $R$ is a homomorphic image of a Gorenstein ring of finite Krull dimension to the assumption that $R$ is CM-excellent (by Remark \ref{22}(4)).
\end{enumerate}
\end{rem}

\section{Filtrations by supports and functions on the prime ideals}

In this section, we study sp-filtrations of $\spec R$ by interpreting them as order-preserving functions on $\spec R$.
First of all, let us recall the definitions of an sp-filtration and the weak Cousin condition.

\begin{dfn}
Let $\phi:\ZZ\to2^{\spec R}$ be a map.
\begin{enumerate}[(1)]
\item
We say that $\phi$ is a {\em filtration by supports} of $\spec R$, an {\em sp-filtration} of $\spec R$ for short, provided that for all integers $i$ the subset $\phi(i)$ of $\spec R$ is specialization-closed and there is an inclusion $\phi(i)\supseteq\phi(i+1)$.
\item
Suppose that $\phi$ is an sp-filtration of $\spec R$.
We say that $\phi$ satisfies the {\em weak Cousin condition} provided that for all $i\in\ZZ$ and all saturated inclusions $\p\subsetneq\q$ in $\spec R$, if $\q\in\phi(i)$, then $\p\in\phi(i-1)$.
\end{enumerate}
\end{dfn}

We construct a correspondence between maps from $\ZZ$ to $2^{\spec R}$ and maps from $\spec R$ to $\ZZ\cup\{\pm\infty\}$.

\begin{dfn}
Let $\phi:\ZZ\to2^{\spec R}$ and $f:\spec R\to\ZZ\cup\{\pm\infty\}$ be maps.
We set:
\begin{align*}
\F(\phi)(\p)&=\sup\{i\in\ZZ\mid\p\in\phi(i)\}+1\ \ \text{for each}\ \ \p\in\spec R,\ \ \text{and}\\
\Phi(f)(i)&=\{\p\in\spec R\mid f(\p)>i\}\ \ \text{for each}\ \ i\in\ZZ.
\end{align*}
Then we get maps $\F(\phi):\spec R\to\ZZ\cup\{\pm\infty\}$ and $\Phi(f):\ZZ\to2^{\spec R}$.
Note that $\F\Phi(f)(\p)=\sup\{i+1\mid\p\in\Phi(f)(i)\}=\sup\{i+1\mid f(\p)\ge i+1\}=f(\p)$ for each $\p\in\spec R$.
Hence the equality $\F\Phi=\id$ holds.
\end{dfn}

Through the maps constructed above, each sp-filtration of $\spec R$ can be interpreted as an order-preserving map from $\spec R$ to $\ZZ\cup\{\pm\infty\}$, that is, one has the following one-to-one correspondence.

\begin{prop}\label{16}
The assignments $\phi\mapsto\F(\phi)$ and $f\mapsto\Phi(f)$ give mutually inverse bijections
$$
\xymatrix{
{\{\,\text{sp-filtrations of $\spec R$}\,\}}
\ar@<.7mm>[r]^-\F
&
{\{\,\text{order-preserving maps from $\spec R$ to $\ZZ\cup\{\pm\infty\}$}\,\}}.
\ar@<.7mm>[l]^-\Phi
}
$$
\end{prop}

\begin{proof}
Fix an sp-filtration $\phi:\ZZ\to2^{\spec R}$ and an order-preserving map $f:\spec R\to\ZZ\cup\{\pm\infty\}$.

If $f(\p)>i+1$, then $f(\p)>i$.
Hence $\Phi(f)(i)\supseteq\Phi(f)(i+1)$ for all $i\in\ZZ$.
Since $f$ is order-preserving, if $\p\in\Phi(f)(i)$ and $\q\in\V(\p)$, then $f(\q)\ge f(\p)>i$ and $\q\in\Phi(f)(i)$.
Thus, $\Phi(f)$ is an sp-filtration of $\spec R$.

Let $\p\subseteq\q$ be an inclusion in $\spec R$.
Since $\phi(i)$ is specialization-closed, if $\p\in\phi(i)$, then $\q\in\phi(i)$.
Hence $\F(\phi)(\p)=\sup\{i+1\mid\p\in\phi(i)\}\le\sup\{i+1\mid\q\in\phi(i)\}=\F(\phi)(\q)$.
Thus, $\F(\phi)$ is an order-preserving map.

We have $\Phi\F(\phi)(i)=\{\p\in\spec R\mid\F(\phi)(\p)>i\}$.
If $\F(\phi)(\p)>i$, then $\sup\{j\in\ZZ\mid\p\in\phi(j)\}\ge i$, and there is an integer $j\ge i$ such that $\p\in\phi(j)$, which implies $\p\in\phi(j)\subseteq\phi(i)$.
Conversely, if $\p\in\phi(i)$, then $i+1\in\{j+1\mid\p\in\phi(j)\}$ and $i+1\le \F(\phi)(\p)$.
We have shown that $\F(\phi)(\p)>i$ if and only if $\p\in\phi(i)$, and it follows that $\Phi\F(\phi)(i)=\phi(i)$ for all $i\in\ZZ$.
Thus, $\Phi\F=\id$.
Now the proof of the proposition is completed.
\end{proof}

We define a function on $\spec R$ to interpret an sp-filtration of $\spec R$ satisfying the weak Cousin condition.

\begin{dfn}
We say that a map $f:\spec R\to\ZZ\cup\{\pm\infty\}$ is a {\em $t$-function} on $\spec R$ provided that for all inclusions $\p\subseteq\q$ in $\spec R$ one has the inequalities $f(\p)\le f(\q)\le f(\p)+\height\q/\p$.
\end{dfn}

\begin{ex}
\begin{enumerate}[(1)]
\item
Let $f:\spec R\to\ZZ\cup\{\pm\infty\}$ be a map, and let $n$ be an integer.
\begin{enumerate}[\rm(a)]
\item
If $f(\p)=n$ for all $\p\in\spec R$ (i.e., $f$ is a constant function), then $f$ is a $t$-function of $\spec R$.
\item
More generally than (a), if $f$ is a $t$-function on $\spec R$, then so is the map $f+n$ given by $\p\mapsto f(\p)+n$.
\end{enumerate}
\item
If $R$ is catenary, then the map $\p\mapsto\height\p$ is a $t$-function of $\spec R$.
\end{enumerate}
\end{ex}

The $t$-functions on $\spec R$ can be characterized in terms of saturated inclusions of prime ideals.

\begin{prop}\label{17}
A map $f:\spec R\to\ZZ\cup\{\pm\infty\}$ is a $t$-function on $\spec R$ if and only if for all saturated inclusions $\p\subsetneq\q$ in $\spec R$ the inequalities $f(\p)\le f(\q)\le f(\p)+1$ hold.
\end{prop}

\begin{proof}
The ``only if'' part is evident. 
To show the ``if'' part, we take any inclusion $\p\subseteq\q$ of prime ideals of $R$, and put $n=\height\q/\p$.
Then there exists a chain $\p=\p_0\subsetneq\cdots\subsetneq\p_n=\q$ in $\spec R$, and then each inclusion is saturated.
Hence $f(\p_{i-1})\le f(\p_i)\le f(\p_{i-1})+1$ for all $1\le i\le n$.
We get inequalities $f(\p)=f(\p_0)\le f(\p_1)\le\cdots\le f(\p_n)=f(\q)$ and $f(\q)=f(\p_n)\le f(\p_{n-1})+1\le f(\p_{n-2})+2\le\cdots\le f(\p_0)+n=f(\p)+\height\q/\p$.
\end{proof}

There is a 1-1 correspondence between sp-filtrations satisfying the weak Cousin conditions and $t$-functions.

\begin{prop}\label{18}
The mutually inverse bijections $(F,\Phi)$ in Proposition \ref{16} induce mutually inverse bijections
$$
\xymatrix{
{\left\{\,
\begin{matrix}
\text{sp-filtrations of $\spec R$}\\
\text{satisfying the weak Cousin condition}
\end{matrix}
\,\right\}}
\ar@<.7mm>[r]^-\F
&
{\{\,\text{$t$-functions on $\spec R$}\,\}}.
\ar@<.7mm>[l]^-\Phi
}
$$
\end{prop}

\begin{proof}
Fix a saturated inclusion $\p\subsetneq\q$ of prime ideals of $R$.

Let $f$ be a $t$-function on $\spec R$.
Then Proposition \ref{16} implies that $\Phi(f)$ is an sp-filtration of $\spec R$.
Suppose that $\q$ belongs to $\Phi(f)(i)$.
We have $f(\q)\ge i+1$ and $f(\q)\le f(\p)+1$ (by Proposition \ref{17}).
Hence $f(\p)\ge i$, which means that $\p$ belongs to $\Phi(f)(i-1)$.
Therefore, $\Phi(f)$ satisfies the weak Cousin condition.

Let $\phi$ be an sp-filtration of $\spec R$ satisfying the weak Cousin condition.
Proposition \ref{16} implies $\F(\phi)(\p)\le\F(\phi)(\q)$.
We want to show the inequality $\F(\phi)(\q)\le\F(\phi)(\p)+1$.
For this, we may assume $\F(\phi)(\q)\ne-\infty$.

Consider the case $\F(\phi)(\q)\ne\infty$.
There is an integer $n$ with $\q\in\phi(n)$ and $\F(\phi)(\q)=n+1$.
As $\phi$ satisfies the weak Cousin condition, $\p$ is in $\phi(n-1)$.
Hence $n=(n-1)+1\le\F(\phi)(\p)$, and $\F(\phi)(\q)=n+1\le\F(\phi)(\p)+1$.

Next we consider the case $\F(\phi)(\q)=\infty$.
In this case, for every integer $n$ there exists an integer $m\ge n$ such that $\q\in\phi(m)$.
The weak Cousin condition on $\phi$ implies $\p\in\phi(m-1)$.
We observe that $\F(\phi)(\p)=\infty$.

We have shown that $\F(\phi)(\p)\le\F(\phi)(\q)\le\F(\phi)(\p)+1$.
Proposition \ref{17} implies that $\F(\phi)$ is a $t$-function on $\spec R$.
Applying Proposition \ref{16} again, we obtain the desired mutually inverse bijections.
\end{proof}

\section{Classification of $t$-structures of $\db(R)$}

The purpose of this section is to classify, as an application of Theorem \ref{11}, the $t$-structures of the bounded derived category $\db(R)$ of finitely generated modules over any CM-excellent ring $R$ of finite Krull dimension.
We start by recalling the definitions of a $t$-structure of a triangulated category, and its aisle and coaisle.

\begin{dfn}
Let $\T$ be a triangulated category.
A {\em $t$-structure} of $\T$ in the sense of Be\u{\i}linson, Bernstein and Deligne \cite{BBD} is by definition a pair $(\X,\Y[1])$ of full subcategories of $\T$ such that $\X[1]\subseteq\X$, $\Y[1]\supseteq\Y$, $\Hom_\T(\X,\Y)=0$ and each object $T\in\T$ admits an exact triangle $X\to T\to Y\rightsquigarrow$ in $\T$ with $X\in\X$ and $Y\in\Y$.
Then $\X$ and $\Y$ are called the {\em aisle} and the {\em coaisle} of the $t$-structure $(\X,\Y[1])$, respectively.
\end{dfn}

Here we recall a couple of basic facts concerning aisles of $t$-structures of a triangulated category.

\begin{rem}
Let $\T$ be a triangulated category.
Then the following statements hold.
\begin{enumerate}[(1)]
\item
A full subcategory $\X$ of $\T$ is the aisle of some $t$-structure of $\T$ if and only if $\X[1]\subseteq\X$, $\X$ is extension-closed (i.e., for any exact triangle $A\to B\to C\rightsquigarrow$ in $\T$ with $A,C\in\X$ one has $B\in\X$), and the inclusion functor $\X\to\T$ has a right adjoint.
For the details, we refer the reader to \cite{KV}.
\item
For each aisle $\X$ in $\T$ there uniquely exists a $t$-structure of $\T$ whose aisle coincides with $\X$.
Thus, classifying the $t$-structures of $\T$ is equivalent to classifying the aisles in $\T$.
\end{enumerate}
\end{rem}

We establish a key lemma to prove our classification theorem of $t$-structures.
We should mention that the lemma is obtained just by virtue of our Theorem \ref{11}, i.e., Faltings' annihilator theorem for complexes.

\begin{lem}\label{19}
Let $R$ be CM-excellent.
Let $\phi$ be an sp-filtration of $\spec R$ satisfying the weak Cousin condition.
Let $n\in\ZZ$ and $X\in\db(R)$.
If $\H_{\phi(i)}^{\le i}(X)=0$ for all $i<n$, then $\H_{\phi(n)}^n(X)$ is a finitely generated $R$-module.
\end{lem}

\begin{proof}
We claim that $\height\p/\q+\depth X_\q>n$ for all $\p,\q\in\spec R$ with $\phi(n)\ni\p\supseteq\q\notin\phi(n)$.
Indeed, there is a chain $\q=\p_0\subsetneq\cdots\subsetneq\p_u=\p$ in $\spec R$, where $u=\height\p/\q$.
The inclusion $\p_{i-1}\subsetneq\p_i$ is saturated for each $1\le i\le u$, and $\p_u=\p\in\phi(n)$.
As $\phi$ satisfies the weak Cousin condition, we inductively get $\q=\p_0\in\phi(n-u)$.
If $u=0$, then $\phi(n)\ni\p=\q\notin\phi(n)$, a contradiction.
Hence $u>0$, and $n-u<n$.
By assumption, we have $\H_{\phi(n-u)}^{\le(n-u)}(X)=0$.
Lemma \ref{15}(2) shows $\depth X_\q>n-u$.
Thus, $\height\p/\q+\depth X_\q>u+(n-u)=n$.

It follows from the above claim and Theorem \ref{11} that there exists an ideal $I$ of $R$ such that $\V(I)\subseteq\phi(n)$ and $I\H_{\phi(n)}^{\le n}(X)=0$.
By \cite[Corollary 3.3]{DZ}, the $R$-module $\H_{\phi(n)}^i(X)$ is finitely generated for every $i\le n$.
\end{proof}

To state our theorem, we recall the definitions of certain maps introduced in \cite{AJS}.

\begin{dfn}[{\cite[paragraph just before Theorem 3.11]{AJS}}]
Let $\phi$ be a sp-filtration of $\spec R$, and $\X$ an aisle in $\db(R)$.
We define the full subcategory $\A(\phi)$ of $\db(R$) and the sp-filtration $\Psi(\X)$ of $\spec R$ by
\begin{align*}
\A(\phi)&=\{X\in\db(R)\mid\supp\H^i(X)\subseteq\phi(i)\text{ for all }i\in\ZZ\},\\
\Psi(\X)(i)&=\{\p\in\spec R\mid(R/\p)[-i]\in\X\}\text{ for each }i\in\ZZ.
\end{align*}
\end{dfn}

Now we can achieve the purpose of this section; the main result of the section is the theorem below.

\begin{thm}\label{12}
If $R$ is CM-excellent and has finite Krull dimension, then one has one-to-one correspondences
$$
\xymatrix{
{\left\{\,
\begin{matrix}
\text{aisles}\\
\text{in $\db(R)$}
\end{matrix}
\,\right\}}
\ar@<.7mm>[r]^-\Psi
&
{\left\{\,
\begin{matrix}
\text{sp-filtrations of $\spec R$}\\
\text{satisfying the weak Cousin condition}
\end{matrix}
\,\right\}}
\ar@<.7mm>[l]^-\A
\ar@<.7mm>[r]^-\F
&
{\left\{\,
\begin{matrix}
\text{$t$-functions}\\
\text{on $\spec R$}
\end{matrix}
\,\right\}}.
\ar@<.7mm>[l]^-\Phi
}
$$
\end{thm}

\begin{proof}
Proposition \ref{18} yields the mutually inverse bijections $(\F,\Phi)$.
The mutually inverse bijections $(\Psi,\A)$ are obtained by Lemma \ref{19} and the proof of \cite[Theorem 6.9]{AJS}.
In fact, the proof of \cite[Theorem 6.9]{AJS} shows the assertion of Lemma \ref{19} under the stronger assumption that $R$ admits a dualizing complex (to invoke the local duality theorem), uses induction on the length of $\phi$ (induction is possible because the existence of a dualizing complex implies the finiteness of the Krull dimension of $R$ by \cite[Chapter V, Corollary 7.2]{H} or \cite[Corollary 1.4]{K0}), and then applies \cite[Lemma 5.7]{AJS}.
This argument remains valid as long as $R$ is a CM-excellent ring of finite Krull dimension.
\end{proof}

Finally, we make a remark to say about the relationship of the above theorem with a result in \cite{AJS}.

\begin{rem}
The 1-1 correspondence of (1) and (2) in \cite[Corollary 6.11]{AJS} for $\db(R)$ is the same as the one $(\Psi,\A)$ in Theorem \ref{12} under the stronger assumption that $R$ admits a dualizing complex; see Remark \ref{22}(4).
\end{rem}

\begin{ac}
The author thanks Hiroki Matsui for valuable comments on Remark \ref{22}(1) and useful discussions about Theorems \ref{11} and \ref{12}.
The author also thanks Leo Alonso Tarr\'{i}o and Ana Jerem\'{i}as L\'opez for helpful comments on the proof of \cite[Theorem 6.9]{AJS}.
\end{ac}

\end{document}